\documentclass[11pt]{article}
\usepackage{amssymb}
\usepackage{amsmath}
\usepackage{amsthm}
\usepackage[textwidth=16cm, textheight=24cm]{geometry}

\usepackage[english]{babel}
\usepackage[utf8]{inputenc}
\usepackage[T1]{fontenc}
\usepackage{enumitem}
\usepackage[all]{xy}
\usepackage{xcolor}
\usepackage{hyperref}
\hypersetup{
    colorlinks,
    linkcolor={red!80!black},
    citecolor={blue!80!black},
    urlcolor={blue!80!black}
}

\newtheorem{thm}{Theorem}
\newtheorem{corollary}[thm]{Corollary}

\newtheorem{lem}[thm]{Lemma}

\newtheorem{prop}[thm]{Proposition}
\newtheorem{defn}[thm]{Definition}
\newtheorem{remark}[thm]{Remark}
\newtheorem{example}[thm]{Example}
\newtheorem{problem}[thm]{Problem}

\newcommand{\Spec}{\operatorname{Spec}}
\newcommand{\Sym}{\operatorname{Sym}}
\newcommand{\Proj}{\operatorname{Proj}}
\newcommand{\tensor}{\otimes}
\newcommand{\PP}{{\mathbb{P}}}
\newcommand{\II}{{\mathcal{I}}}
\newcommand{\OO}{{\mathcal{O}}}
\newcommand{\kk}{\mathbf{k}}%
\newcommand{\mm}{\mathfrak{m}}%
\newcommand{\Aspace}{\mathbb{A}^{n}}%
\newcommand{\Aspacecomp}{\PP^{n}}%
\newcommand{\HilbdA}{\operatorname{Hilb}_d}%
\newcommand{\HilbdAsmooth}{\operatorname{Hilb}_d^{\circ}}%
\newcommand{\HilbdAzero}{\operatorname{Hilb}_d^{sm}}%
\newcommand{\Apolar}[1]{\operatorname{Apolar}(#1)}%
\newcommand{\FBS}[1]{\operatorname{Bs}_{\operatorname{aff}}(#1)}%
\newcommand{\Ann}[1]{\operatorname{Ann}\left( #1 \right)}%
\newcommand{\spann}[1]{\operatorname{span}\left( #1 \right)}%

\DeclareMathOperator{\Gr}{Gr}
\DeclareMathOperator{\Bl}{Bl}
\DeclareMathOperator{\Sec}{Sec}

\newcommand{\Gproj}{\Gr(3, \Sym^2 V)}%
\newcommand{\Gshort}{\Gr}%
\newcommand{\cT}{\mathcal{T}}%
\newcommand{\cQ}{\mathcal{Q}}%

\begin{document}

\title{Smoothable zero dimensional schemes and special projections of algebraic varieties}
\author{Joachim Jelisiejew, Grzegorz Kapustka, Micha{\l} Kapustka}
\maketitle

\begin{abstract} We study the degrees of generators of the ideal of a projected Veronese variety
$v_2(\PP^3)\subset \PP^9$ to $\PP^6$  depending on the center of projection.
This is related to the geometry of zero dimensional schemes of length
$8$ in $\mathbb{A}^4$, Cremona transforms of $\PP^6$, and the geometry of Tonoli  Calabi-Yau threefolds of degree $17$ in $\mathbb{P}^6$.
\end{abstract}
{\small\textbf{MSc classes:} primary: 14J32, 14E07, 14C05, secondary: 14J10.}

\section{Introduction}

The aim of the paper is to find and investigate a relation between the following three a priori distinct subjects:
\begin{itemize}
\item analysing smoothability of finite degree $8$ subschemes in $\mathbb{A}^4$.
\item describing  special projections of the double Veronese embedding of
    $\PP^3$ to $\PP^6$,
\item studying the action of special $(2,4)$ Cremona transformations
    $\PP^6 \dashrightarrow \PP^6$,
\end{itemize}
For a linear subspace $L \subset\Sym^2 V$ let us
denote by $\pi_L\colon \PP(\Sym^2 V)\dashrightarrow \PP(\Sym^2 V/L)$ the
projection. Let $v_2\colon \PP(V)\to \PP(\Sym^2 V)$ be the second Veronese
embedding. We prove the following theorem.
\begin{thm}\label{ref:mainintro:thm}
    Let $V = \mathbb{A}^4$ and $L \subset \Sym^2 V$ be a linear subspace of dimension
  $3$. Assume that the composition $\pi_L \circ v_2\colon \PP(V)
  \dashrightarrow \PP(\Sym^2 V/L)$ is regular and an embedding. Let $X_L$ be its image.  Then the following are equivalent
  \begin{enumerate}[label=(\alph*)]
          \item\label{it:maincremona} $\pi_L \circ v_2:\PP^3\to \PP^6$ is a
              restriction of
              a Cremona transformation $\PP^6 \dashrightarrow \PP^6$ of type
              $(2,4)$ based in a rational octic surface or a limit of such
              restrictions,
          \item\label{it:mainjump} $X_L$ is contained in a cubic hypersurface,
          \item\label{it:mainjumpthree} $X_L$ is contained in a three-dimensional
              space of cubic hypersurfaces,
          \item\label{it:mainsmoothable} the scheme $R = \Spec \Apolar{L}$ is smoothable,
          \item\label{it:maindescription} $L$ is spanned by partial derivatives of a cubic form $F \in
            \Sym^3 V$.
    \end{enumerate}
    The above equivalent conditions describe a closed, irreducible subset of
    $\Gr(3, \Sym^2 V)$. For a general $L$ in this subset the space of cubics
    containing $X_L$
    is exactly three-dimensional and it is spanned by Segre cubics.
    \end{thm}


Our original motivation for studying these projections is related to
understanding the geometry of constructions of Calabi--Yau threefolds in
$\PP^6$, some of which are related to special projections of the
Veronese embedding of $\PP^3$. More precisely,  in \cite{KK, KK2}, the
second and third author study the projections
$$\pi_{L}(v_2(\PP^3))=X_{L}^8\subset \PP^6$$ from
$\PP(L)$ for $L$ of dimension 3. For a generic $L$ one has $H^0(\II_{X^8_{L}}(3))=0$. 
However, using \emph{Macaulay2} with a lot of random choices it was proven that one can find $L$ 
 such that $X^8_{L}$ is smooth and $H^0(\II_{X^8_{L}}(3))=3$. 
 For such special projections,  by the bilinkage construction one can construct degenerated Tonoli Calabi-Yau threefolds of degree $17$ in $\PP^6$ (cf.~\cite{KK}).
 In the
 present paper we explain the geometric meaning of these special exceptional centers.
Our problem is related to the following more general subject.
Let $X \subset \PP(W)\simeq \PP^r$ be an algebraic variety,
 consider the projection $\pi_{L}\colon X\to X_{L} \subset \PP^{r-t}$ from the center $$\PP(L)=\PP^{t-1}\subset \PP(W)$$ such that $\pi_{L}$ is an isomorphism. 
 \begin{problem} What are the possible Betti numbers of the ideal of
     $X_{L}\subset \PP^{r-t}$ when we move the center of projection $L\in \Gr(t,W)$?
\end{problem}

The study of the geometry of central projections is a classical topic that was widely studied for generic projections \cite{R}, \cite{GrusonPeskine}, \cite{BEisenbud}.
The study of Betti numbers of projected varieties was discussed in \cite{AK}, \cite{HK}, \cite{AR}.
Using the mapping cone construction the authors were able under some
conditions to relate the Betti numbers of the variety before and after the
projection. In particular in \cite[Prop.~4.11]{AK} it is described how the
number of quadrics in the ideal of a centrally projected variety changes when
we move the center of projection.

The case of projections with center being a point is also considered in \cite{AR}.
The authors describe sub-schemes $Z_k(X)\subset \PP^r$ being the loci of points such that the ideal of the projected variety admits more generators of a given degree. 
The aim of this paper is to study the case when the center of projection has dimension $\geq 1$.
The first case to consider are the projections of the second Veronese embeddings $v_2(\PP^2)\subset \PP^5$ and its projections to smooth surfaces in $\PP^4$. However, it is easy to see that all such projections give projectively isomorphic images. The cases of projection of $v_2(\PP^3)$ to $\PP^8$ and $\PP^7$ are treated in \cite{AR}. Our situation is hence the next case to check.

The problem is unexpectedly related  to the theory of zero-dimensional
schemes, which
seems a purely algebraic one. For such a scheme $R\subset \Aspace$ we say
that $R$ is \emph{smoothable} if it is a limit of smooth subschemes (tuples of
points). More precisely, let $\HilbdA$ be the Hilbert scheme of $d$ points on
$\Aspace$, so that (closed) points of $\HilbdA$ are finite degree $d$
subschemes of $\Aspace$. Let $\HilbdAsmooth \subset \HilbdA$ be an open subset
consisting of tuples of $d$ distinct points on $\Aspace$. Let
$\HilbdAzero$ be the closure of $\HilbdAsmooth$. Then $R \subset \Aspace$ is
smoothable if and only if the corresponding point lies in $\HilbdAzero$.
Whether a given $R$ is smoothable is a difficult question, see~\cite{CEVV,
erman_velasco_syzygetic_smoothability, cn09, cn10, cjn13, DJNT,
Sivic__Varieties_of_commuting_matrices,
bertone_cioffi_roggero_marked_bases,
bertone_cioffi_roggero_smoothable_Gorenstein, dosSantos}. It is connected
with the search for equations of secant varieties~\cite{bubu2010,
jabu_ginensky_landsberg_Eisenbuds_conjecture}. An important aspect of
Theorem~\ref{ref:mainintro:thm} is that it gives a geometrical interpretation
of smoothability, in a special case.


Our special projections are also related to the theory of Cremona
transformations, i.e., birational self maps of the projective space. Such self
maps are induced by systems of homogeneous polynomials of equal degree called
the \emph{degree} of the Cremona transformation. These polynomials define the
indeterminacy of the map called the \emph{base} of the Cremona transformation. One also defines the \emph{type} of a Cremona transformation as a pair consisting of its degree and the degree of its inverse. 
The problem of classification of Cremona transformations with smooth base loci was considered in
 \cite{Crauder_Katz__Cremonas_with_smooth_fundamental_locus, PirioRusso,
hulek_katz_schreyer__Cremonas_and_syzygies, Stagliano}. In particular,
Cremona transformations with base loci being smooth surfaces was classified
in \cite{Crauder_Katz__Cremonas_with_smooth_fundamental_locus}. In this work we investigate and exploit the geometry of one of the five types
of transformations with base loci being smooth surfaces:  Cremona transformations of type
$(2,4)$ in $\PP^6$ based in a rational surface embedded as a surface of degree
$8$ and sectional genus $3$  in $\PP^6$.  

 

The paper is centered around the proof of Theorem~\ref{ref:mainintro:thm}. In
Section~\ref{sec:prelims} we discuss general results on base loci and
zero-dimensional schemes. Section~\ref{sec:projs} applies them to the special
case of quadric embeddings $\PP^3\to \PP^6$. In
Subsection~\ref{ssec:HilbertEight} we discuss deformations of
zero-dimensional, degree $8$ subschemes of $\mathbb{A}^4$ and give a geometric
proof of the equivalence of Condition~\ref{it:maindescription}
and Condition~\ref{it:mainsmoothable} from Theorem~\ref{ref:mainintro:thm} (the
equivalence was first proven in~\cite{CEVV}). In
Subsection~\ref{ssec:schreyerfest} we prove
equivalence of Conditions~\ref{it:mainsmoothable}, \ref{it:mainjump}
and~\ref{it:mainjumpthree}.
Finally, in Subsection~\ref{ssec:Cremonas} we prove equivalence
of Conditions~\ref{it:maincremona} and~\ref{it:mainsmoothable} and provide a
formal proof of Theorem~\ref{ref:mainintro:thm}.


\section{Preliminaries}\label{sec:prelims}
We work over an algebraically closed base field $\kk$ of characteristic $\neq2, 3$.
For a vector space $V$, by $\PP(V)$ we mean the scheme $\Proj \Sym(V^\vee)$.
Then the cone over $\PP(V)$ is identified with $V$.
When speaking about a rational map (or a morphism) ${\varphi:\PP(V)
\dashrightarrow \PP(W)}$
we always implicitly fix a morphism $\hat{\varphi}:V\to
W$ inducing $\varphi$.
\begin{defn}\label{ref:jumplocus:def}
        We define the \emph{jump locus} inside the Grassmannian
        $J^l_k\subset \Gr(t,r+1)$ by the following \emph{jump condition}:
         a point $L \in \Gr(t,r+1)$ is in the set $J_k^l$ if
           $H^0(\II_{X_L}(k))$
        has dimension higher then the generic value and the difference is $l$.
    \end{defn}
  By the semicontinuity theorem we deduce that $l\geq 0$.

It follows from \cite[Prop.~4.1]{AK} that the isomorphic projection of a $m$-normal variety from a center
of dimension $t-1$ is still $m$-normal for $m\geq t+1$. In the case where $ X\subset \PP^r$ is projectively normal it follows from \cite[Prop.~2.1]{AR} that the number of hypersurfaces of degree $k\geq t+1$ in the ideal of the projected variety is uniquely determined. So we are interested in the sets $J^l_k$ for $k\leq t$.  

On the other hand from \cite[Prop. 4.11]{AK}, the number of generators of degree $2$ of the centrally projected variety (from a point) is uniquely determined by the dimension of the secant locus of the projection in the case $X\subset \PP^r$. 

Natural problems occur:
\begin{itemize}
\item Are the loci $J_k^l$ related to the secant loci of $X\subset \PP^r$,
\item For a given $X\subset \PP^r$ and $2<k\leq t$ what are the possible values of $l$ such that $J^l_k$ is non-empty.
\item Are there jumps i.e $0<l<p<m$ such that $J^l_k\neq \emptyset$ and $J^m_k\neq \emptyset$ but
$J^p_k= \emptyset$ for some $k$.

\end{itemize}
In this paper we address all those questions in our example.
 
Note that we cannot describe $J^l_k$ directly by induction using projections from points studied in \cite{AR} since the schemes $Z_k(X_p)$ vary when we move the center $p\in \PP^r-\Sec(X)$.

\subsection{Affine base loci}

One of the main objects in our study of rational maps are the affine base loci,
that we define below. Recall that in our convention a rational map
$\varphi\colon \PP(V)\dashrightarrow \PP(W)$ comes with a fixed map
$\hat{\varphi}\colon V\to W$ on the level of cones.
\begin{defn}
    Let $\varphi:\PP(V)\dashrightarrow \PP(W)$ be a rational map.
    The \emph{affine base locus} is the
    affine scheme $\hat{\varphi}^{-1}(0) \subset V$. We denote it by
    $\FBS{\varphi}$.
\end{defn}
For every $\varphi$, the affine base locus is invariant under homothety and its image
in $\PP(V)$ is the base locus for $\varphi$. If $\varphi$ is
regular, then $\FBS{\varphi}$ is supported at $0 \in V$, hence it is
zero-dimensional. Note that $\FBS{\varphi}$ may be non-empty even for
a regular $\varphi$.

Recall that each rational map between projective spaces is given by a $d$-th
Veronese embedding composed with a
linear projection and that such map is regular (respectively, isomorphism onto the image) if
and only if the center of the linear projection does not intersect the image of
$\PP(V)$ in $\PP(\Sym^d V)$ (respectively, the secant variety of the
image).

\begin{example}
    Let $\PP^3\to \PP^6$ be a morphism given by seven quadrics. The algebra $A
    = H^0(\OO_{\FBS{\varphi}}) = \kk[x_0, x_1, x_2, x_3]/(q_1, \ldots ,q_7)$ is
    zero-dimensional and graded. For a general enough choice of quadrics we
    have $(x_0, x_1, x_2, x_3)(q_1, \ldots ,q_7) = (x_0, x_1, x_2, x_3)^3$,
    hence $A$ is spanned by unity, linear forms and three complementary
    quadrics, so it has degree $1 + 4 + 3$.
\end{example}

We now aim at describing the geometry behind the affine base locus of a
\emph{morphism} of projective spaces. This sends us to the world of finite
schemes and Hilbert schemes of points.

\subsection{Apolarity}\label{ssec:apolarity}

Recall that we have assumed that $\kk$ has characteristic not equal to two or
three.
For a characteristic-free description of apolarity and
further information see e.g.~\cite{iakanev, Jel_classifying}.

We recall a very useful parameterization tool for finite schemes, called
apolarity or Macaulay inverse systems. Namely, $\Sym V^\vee$ acts on $\Sym V$,
where elements of $V^\vee$ act as partial derivatives.
For a finite dimensional subspace $L \subset \Sym V$ we may consider the ideal
$\Ann{L}$ of all operators from $\Sym V^\vee$ annihilating $L$ and the
quotient
\[
    \Apolar{L} = \Sym V^\vee/\Ann{L}
\]
which is a local zero-dimensional $\kk$-algebra with residue field $\kk$ and
of rank equal to $\dim_{\kk} (\Sym V^{\vee}\circ L)$.
We will be mostly interested in the case when $\dim V = 4$ and $L$ is a
three-dimensional space of quadrics.

    \begin{example}
        Let $V = \spann{x, y, z, t}$ and $L = \spann{x^2, y^2,
        z^2 - t^2}$. Let  $V^\vee = \spann{\partial_x, \partial_y, \partial_z,
        \partial_t}$ be the dual basis.
        Then
        \[\Ann{L} = \spann{\partial_x\partial_y,\ \partial_x\partial_z,\
            \partial_x\partial_t,\ \partial_y\partial_z,\ \partial_y\partial_t,\
        \partial_z\partial_t,\ \partial_z^2 + \partial_t^2} + \Sym^{\geq 3} V^\vee.\]
        Consequently, $\Apolar{L}  \simeq  \spann{\partial_x^2,\ \partial_y^2,\
        \partial_z^2 - \partial_t^2,\ \partial_x,\ \partial_y,\ \partial_z,\
        \partial_t, 1}$ as linear spaces.
    \end{example}

A theorem of Macaulay asserts that $L \mapsto \Apolar{L}$ induces a bijection between
zero-dimensional subschemes of $V$ supported at the origin and finite
dimensional subspaces $L' \subset \Sym V$ which are closed under the action of
$\Sym V^{\vee}$, i.e., which are $(\Sym V^{\vee})$-submodules. Clearly subspaces spanned by homogeneous
elements give $\kk^*$-invariant schemes and conversely. Moreover, principal $(\Sym V^{\vee})$-submodules
correspond precisely to Gorenstein schemes. For example, a general cubic
gives a graded Gorenstein subscheme with Hilbert function $(1, n, n, 1)$, where $n =
\dim V$. Note that $\Apolar{-}$ is order preserving: a subscheme corresponds to a
smaller linear space.

\subsection{Geometry of zero-dimensional schemes}

For a zero-dimensional scheme, its \emph{length} is the linear dimension of
its algebra of global sections. If the scheme is embedded into a projective
space, then it is the same as its degree.  In this subsection we fix the
length $d$ of considered schemes. If the scheme is irreducible, then it
corresponds to a \emph{local} algebra $(A, \mm, \kk)$. In this case by the \emph{Hilbert
function} we denote $H(i) = \dim_{\kk} \mm^i/\mm^{i+1}$. This function is
usually written as a vector of its non-zero values, e.g.~for a first
infinitesimal neighbourhood of a point in $\mathbb{A}^3$ one gets $(1, 3)$. For
a general reference on finite schemes,
see~\cite{jabu_jelisiejew_smoothability}.

For a general reference on the Hilbert schemes of points,
see \cite{fantechi_et_al_fundamental_ag, haiman__multigraded,
hartshorne_deformation_theory}.

The functor of embedded flat families of length $d$ zero-dimensional subschemes of
$\Aspace$ is represented by the \emph{Hilbert scheme of points} of
$\Aspace$, which we denote $\HilbdA \Aspace$ or shortly $\HilbdA$. It is a connected,
quasi-projective scheme; in fact it is an open subset of the Hilbert scheme of
points on $n$-dimensional projective space (working with an affine instead
of a projective space is natural, since we look towards analysing affine base
loci). Closed points of $\HilbdA$ correspond to zero-dimensional
schemes; we will denote by $[R]$ the
point corresponding to a subscheme $R \subset \Aspace$.

The most natural zero-dimensional subscheme of $\Aspace$ of length
$d$ is just a $d$-tuple of points. Denote by $\HilbdAsmooth \subset \HilbdA$
the subset corresponding to all tuples of points. A scheme is a tuple of points
precisely when it is smooth, hence $\HilbdAsmooth$ is open in $\HilbdA$.
Its closure is then an irreducible component of $\HilbdA$, denoted by
$\HilbdAzero$. The points of $\HilbdAzero$ are limits of smooth schemes, hence are called
\emph{smoothable} schemes and $\HilbdAzero$ is called the \emph{smoothable
component}. This component has dimension $nd$.

Smoothability has a down-to-earth characterisation, at least in the graded
case and generically. For a scheme $R \subset \Aspace$ we say that it is a
\emph{$\kk^*$-limit} if $R = \lim_{t\to 0} t\Gamma$ for a tuple $\Gamma$ of $d$
points of $\Aspace$. Then $R$ is $\kk^*$-invariant and smoothable.
The notion of $\kk^*$-limits may be formulated differently as follows.
Fix a $\kk^*$-limit $R = \lim_{t\to 0} t\Gamma$.
Compactify $\Aspace$ to a $\Aspacecomp$ by adding a ``time coordinate''. Then
$\Gamma$ becomes as a set of $d$ points of $\Aspacecomp$ and $R$ is the
hyperplane section of the cone over $\Gamma$ by the hyperplane corresponding
to $t = 0$.
We say that an irreducible scheme $R \subset \Aspace$ is
\emph{compressed} if there is an $s$ such that the Hilbert function of $R$
satisfies $H(i) = \dim \Sym^i \kk^{n}$ for all $i < s$ and $H(i) = 0$ for
all $i > s$. For example if $R \subset \mathbb{A}^3$ is an irreducible scheme of
length $12$, then it is compressed if and only if its Hilbert function is $(1,
3, 6, 2)$.

\begin{corollary}\label{ref:smoothabilityofgraded:cor}
    The set of smoothable compressed subschemes of $\Aspace$ is irreducible
    and its general member is a $\kk^*$-limit.
\end{corollary}
\begin{proof}
    This follows from~\cite[Lemma 5.4]{CEVV}.
\end{proof}

\begin{corollary}\label{ref:extendingrationas:cor}
    Let $\varphi:\PP(V)\to \PP(W)$ be a morphism and $R \subset V$ be its affine base
    locus. If $R$ is a $\kk^*$-limit then there exists an
    inclusion $V \subset V'$ with one dimensional cokernel and an extension
    $\psi:\PP(V')\dashrightarrow \PP(W)$. Conversely, if $\psi$ exists and
    its base locus $\Gamma$ is non-empty and smooth of degree $\deg R$, then $R$ is $\kk^*$-limit.
\end{corollary}
\begin{proof}
    If $R \subset V$ is a $\kk^*$-limit of $\Gamma$ then one may compactify
    $V$ to $\PP(V')$ by adding a ``time coordinate'' as above. Then $R$ is a
    hyperplane section of the cone over $\Gamma$ and the equations of
    $\varphi$ defining $\Gamma$ lift to equations of $\Gamma$, which induce a
    rational map $\psi:\PP(V')\dashrightarrow \PP(W)$ with base locus $\Gamma$.
    Conversely, if $\psi$ exists and its base locus $\Gamma$ is smooth non empty, then it is
    also zero-dimensional. Let $\hat{\Gamma}$ be the affine base locus of
    $\psi$. Then $\hat{\Gamma}$ is cone over $\Gamma$, perhaps with some embedded component
    at the origin (due to lack of saturation). Let $\hat{\Gamma}' \subset \hat{\Gamma}$ be the cone given
    by saturation $I(\Gamma)$. The scheme $R' = \hat{\Gamma}' \cap V$ is
    a hyperplane section of $\hat{\Gamma}'$, hence has degree $\deg \Gamma$, which is equal to $\deg
    R$ by assumption. The scheme $R'$ is a $\kk^*$-limit
    of the affine scheme $\Gamma$ in $\PP(V')\setminus\PP(V)$. By definition, $R = V\cap
    \hat{\Gamma}$, thus $R' \subset R$ are two zero-dimensional schemes of
    the same degree, so $R' = R$.
\end{proof}

For $n\leq 2$ the smoothable component is the unique component and in fact
$\HilbdA$ is smooth~\cite{fogarty}. This is no longer the case for $n \geq 3$.
If the dimension is at least three and $d$ is large enough,
then the Hilbert scheme is reducible and singular (\cite{CEVV,
erman_Murphys_law_for_punctual_Hilb, iaCompressed} and also
\cite{emsalem_iarrobino_small_tangent_space}\footnote{Note: there is a known numerical
    mistake in the computation on page 169, compare~\cite{cn09}.}). Not much is known about the
additional components of the Hilbert scheme, in fact their mere presence seems
to discourage investigators.

It is known that for $d\leq 7$, any $n$ or for $d = 8$, $n\leq 3$ the Hilbert
scheme is irreducible~\cite{CEVV}.
This scheme is reducible for $d = 8$ and $n \geq 4$.
In this paper we are interested in the ``smallest'' reducible example: the
Hilbert scheme of $d = 8$ points on $\mathbb{A}^4$.

\section{Projections of $v_2(\PP^3)$}\label{sec:projs}
In this section we study the geometry of projections of $v_2(\PP^3)\subset \PP^9$ to $\PP^6$.
Let us introduce some useful tools.
\subsection{Hilbert scheme of eight points on affine
four-space}\label{ssec:HilbertEight}

    \newcommand{\Hilbonefourthree}{\operatorname{Hilb}_{143}}%

    The paper~by Cartwright, Erman, Velasco and Viray~\cite{CEVV} is devoted
    to the  analysis of the Hilbert scheme of eight points on affine $4$-space.
    Most of the facts below are found there; our contributions are
    Proposition~\ref{ref:cubicvssmoothable:prop} and
    Remark~\ref{ref:mojezeby:rmk}.

    The Hilbert scheme of $8$ points on $V := \mathbb{A}^4$ has two irreducible
    components. The smoothable component has dimension $8\cdot 4 = 32$.
    The other component, $\Hilbonefourthree$, has dimension $25$ and is
    isomorphic to $V \times \Gr(7, \Sym^2 V^\vee)$. The isomorphism is given
    by sending $(p, N)$ to the irreducible scheme $\Sym V^\vee/(N) + (V^\vee)^3$
    translated so that its support is $p$. Hence the schemes
    corresponding to points of $\Hilbonefourthree$ are irreducible and have
    Hilbert function $(1, 4, 3)$. By apolarity we may equivalently
    parameterise this component as $V \times \Gr(3, \Sym^2 V)$, by sending $(p,
    L)$ to $\Apolar{L + V}$ supported at $p$. Note that if $\Apolar{L+V}\neq
    \Apolar{L}$, then the partials of forms of $L$ span a proper subspace
    $W$ of $V$,
    hence $L \subset \Sym^2 W$. In this case $\PP(L)$
    intersects the secant variety of $v_2(\PP(W))$. Thus, for our purposes,
    the difference between $\Apolar{L}$ and $\Apolar{L+V}$ is negligible.

    The two components intersect along an irreducible $24$-dimensional set,
    which has the form $V \times D$, where $D \subset \Gr(3, \Sym^2
    V)$ is a divisor of degree two on the Grassmannian,
    which we call the \emph{smoothable divisor}.
    Hence smoothability is independent of the embedding (this is true for all
    finite schemes~\cite[Theorem~1.1]{jabu_jelisiejew_smoothability}).
    The equation of $D$ is obtained as
    follows: let $W \subset \Sym^2 V$ be a $3$-dimensional space spanned by
    $q_1, q_2, q_3$ that correspond to $4\times 4$ matrices $A_1$, $A_2$,
    $A_3$.  Then the equation is the Pfaffian of the $12 \times 12$ matrix
    \begin{equation}\label{eq:smoothableDivisorEquation}
        \begin{bmatrix}
            0    & A_1 & -A_2\\
            -A_1 &  0  &  A_3\\
            A_2  & -A_3&  0
        \end{bmatrix}.
    \end{equation}

    Salmon gave a geometric description of the smoothable divisor, by showing
    that its equation vanishes on $q_1, q_2, q_3$ if and only if there exists a cubic $F$
    and linear differential operators $d_1, d_2, d_3$ such that $d_iF = q_i$.
    Below we give another proof of this fact.

    \begin{prop}\label{ref:cubicvssmoothable:prop}
        Let $L\subset \Sym^2 V$ be $3$-dimensional
        and $R = \Spec \Apolar{L + V}$ be the corresponding
        zero-dimensional scheme of degree eight.
        The following conditions are equivalent:
        \begin{enumerate}
            \item\label{it:special} there is a cubic $F$ such that
                $L$ is spanned by three partial derivatives of~$F$,
            \item\label{it:smoothable} the scheme $R$ is smoothable.
        \end{enumerate}
    \end{prop}

    \begin{proof}
        The implication~\ref{it:special} $\implies$ \ref{it:smoothable} is
        noted in \cite[Rmk~5.9]{CEVV} and we refer to this work for details (see also
        Remark~\ref{ref:mojezeby:rmk}).

        We will prove the implication \ref{it:smoothable} $\implies $\ref{it:special}.
        By the obvious parameterisation, the set of $L$ satisfying
        Condition~\ref{it:special} is closed. Thus it is enough to show that each
        point $[R]$ is a limit of points in the smoothable divisor satisfying~\ref{it:special}.
        By Corollary~\ref{ref:smoothabilityofgraded:cor} the set of
        smoothable, irreducible schemes with Hilbert function $(1, 4, 3)$ is
        irreducible and the set of $\kk^*$-limits is dense inside it. Hence we
        may consider only $\kk^*$-limits.

        Let $R = \lim_{t\to 0} t\Gamma$ be a $\kk^*$-limit in $\mathbb{A}^4$
        with Hilbert function $(1, 4, 3)$.
        Compactify $\mathbb{A}^4$ to $\PP^4$ and consider $\Gamma \subset
        \PP^4$.
        By~\cite[Theorem~8.6]{Eisenbud__Popescu__Gale_geometry} this set can
        be enlarged to a tuple $\Gamma'$ of ten arithmetically Gorenstein points of
        $\PP^4$. Since $R = \hat{\Gamma} \cap H$ is a hyperplane section of the cone
        over $\Gamma$, it lies in a hyperplane section $S = \hat{\Gamma'}\cap H$ of the cone over
        $\Gamma'$. This section is zero-dimensional Gorenstein with Hilbert
        series $(1, 4, 4, 1)$, provided that $H$
        intersects $\Gamma'$ properly. This condition can be always achieved
        by perturbing $H$, which perturbs $R$ in the smoothable divisor.
        Hence, at least in any
        neighbourhood of $R$, we get a point $R_{\varepsilon} \subset
        S_{\varepsilon}$, where $S_{\varepsilon}$ is Gorenstein with Hilbert
        function $(1, 4, 4, 1)$.  We conclude that $S_{\varepsilon} =
        \Apolar{F_{\varepsilon}}$ for some cubic and that
        $L_{\varepsilon}$ is contained in the partials of this cubic.
    \end{proof}

    \begin{remark}\label{ref:mojezeby:rmk}
        In the proof of Proposition~\ref{ref:cubicvssmoothable:prop} we got an
        inclusion of $R$ into $S$ whose Hilbert function is $(1, 4, 4, 1)$ and
        even an inclusion of curves whose hyperplane sections are $R$ and $S$.
        Moreover, Condition~\ref{it:special} asserts that for a given $S$ \emph{every} $R
        \subset S$ is smoothable. But for a given $S$ and a curve $C_S$ smoothing
        it as above  $R$ is not necessarily smoothed by a sub-curve $C_R$. Indeed,
        for a fixed $C_S$ the existence of $C_R$ is equivalent to the
        existence of $C_{R'}$, where $R' \subset S$ is the residuum of $R$
        in $S$. In our case $R'$ is a tangent vector in $S$. One sees
        immediately that the tangent vectors which do lift lie in the planes
        spanned by the two of the $10$ points of $C_S$. In particular not all
        tangent vectors lift.
    \end{remark}

    \subsection{Smoothable schemes of degree $8$ and special quadric morphisms
    $\PP^3\to \PP^6$}\label{ssec:schreyerfest}

    Let $V$ have dimension four.
    In this section we consider morphisms $\PP(V)\to \PP^6$ given by quadrics.
    They factor as second Veronese $v_2$ composed with a projection
    \[
        \pi_L\colon \PP(\Sym^2 V)\dashrightarrow \PP^6
    \]
    from an $L \subset \Sym^2 V$. Let $X_L = \pi_L(v_2(\PP^3))$.
    A general $\PP(L)$ does not intersect the secant variety of
    $v_2(\PP^3)$, hence for such $L$ the variety $X_L$ is isomorphic to
    $\PP(V)$ via $\pi_L\circ v_2$.

    Denote by $L^{\perp} \subset \Sym^2 V^\vee$ the space perpendicular to $L$,
    then naturally $\PP^6 = \PP((L^{\perp})^\vee)$.
    The pullback of sections from $\PP^6$ to $X_L$ and then to
    $\PP (V)$ gives a restriction map $H^{0}(\OO_{\PP ^6}(3))\to
    H^0(\OO_{\PP(V)}(3))$, which algebraically reads
    \begin{equation}\label{eq:degenerationOnFiber}
        \Sym^3 L^{\perp} \to \Sym^6 V^\vee.
    \end{equation}
    The spaces on both sides of this map have dimension $84$ and in
    fact for a general choice of $L$ this morphism is an isomorphism.
    Let $\Gshort := \Gproj$ be the parameter space for such $L$.
    The above discussion globalises as follows: the tautological sequence
    \[
        0 \to \cT \to \OO_{\Gshort}\tensor \Sym^2 V \to \cQ\to 0
    \]
    dualises to $0\to \cQ^{\vee} \to \OO_{\Gshort}\tensor \Sym^2 V^{\vee} \to
    \cT^{\vee} \to 0$. The multiplication $\Sym^3 \Sym^2 V^{\vee} \to \Sym^6
    V^{\vee}$ induces a map
    \begin{equation}\label{eq:degenerationGlobal}
        \mu\colon \Sym^3 \cQ^{\vee} \to \OO_{\Gshort} \tensor \Sym^6 V^{\vee},
    \end{equation}
    whose restriction to a point $[\PP(L)]\in \Gshort$ is
    exactly~\eqref{eq:degenerationOnFiber}. Thus the jump locus has a natural
    scheme structure at the degeneracy locus of $\mu$. A Schubert calculation
    shows that the degree of the jump locus is $36$.

    As described in Section~\ref{ssec:apolarity}, a point $[\PP(L)]$ of
    $\Gproj$ has an associated zero-dimensional scheme $R = \Spec \Apolar{L +
    V}$, which is of length $8$ and for a general $L$ it is equal to $\Spec
    \Apolar{L}$ (so that adding $V$ may be thought of as stabilisation of length).

    Now comes the Leitmotiv of this work: we show that
    smoothability implies the jump and even more.
    The authors discussed the problem just before the conference dinner at
    Schreyerfest and the first computational evidence for
    Theorem~\ref{ref:schreyerfestdinner:thm} was
    obtained just before the dessert. Hence we call the following
    \emph{Schreyerfest-dinner theorem}.

    For a three-dimensional $L \subset \Sym^2 V$ we say that \emph{$L$
    satisfies the jump condition} if the image $X_L$ of $v_2\PP(V)$ under the
    projection from $\PP(L)$ lies on a cubic. We say that the jump is at least
    by $l$ if $X_L$ lies on an (at least) $l$-dimensional space of cubics.

    \begin{thm}[Schreyerfest-dinner]\label{ref:schreyerfestdinner:thm}
        Let $L \subset \Sym^2 V$ be a linear subspace of dimension three.
        Suppose that $R = \Spec \Apolar{L + V}$ is smoothable.
        The $\PP(L)$ satisfies the jump condition and the jump is at least by
        three.
    \end{thm}
    \begin{proof}
        \def\Gamfour{\Gamma_8}%
        \def\Gamtwo{\Gamma'_8}%
        The jump condition is closed, so it is enough to prove it for a
        general $L$. Hence, by Corollary~\ref{ref:smoothabilityofgraded:cor}
        we may assume that $R$ is a $\kk^*$-limit and that its ideal is
        generated by quadrics. By
        Corollary~\ref{ref:extendingrationas:cor} the map
        $\pi_L\circ v_2:\PP(V) = \PP^3 \to \PP^6$ extends to a rational map $\PP^4
        \dashrightarrow \PP^6$ whose base locus $\Gamfour$ is a tuple of $8$ points.
        Since a general tuple of such $8$ points gives $R$ as above, we may
        assume that $\Gamfour$ is general. Also the ideal of $\Gamfour$ is
        generated by seven quadrics.

        Now we apply Gale duality. We refer the reader to the beautiful
        paper~\cite{Eisenbud__Popescu__Gale_geometry}.
        Since $\Gamfour \subset \PP^4$ is general, its Gale dual exists by
        \cite[Corollary~2.4]{Eisenbud__Popescu__Gale_geometry} and is a
        $8$-tuple $\Gamtwo \subset \PP^2$, which we may also assume to be
        general. Hence, there exists $q\in \PP^2$ such that $\Gamtwo \cup
        \{q\}$ is a complete intersection of two elliptic curves. Also, in this case the
        Gale duality can be made explicit:
        starting from $\Gamtwo \subset \PP^2$ one takes the second Veronese
        reembedding and projects from $q$, obtaining $\Gamfour \subset \PP^4$,
        see~\cite[Corollary~2.6 and p.~138-9]{Eisenbud__Popescu__Gale_geometry}.
        In any case, the pencil of elliptic curves passing through
        $\Gamtwo$ gives a pencil $\mathcal{E}$ of projectively normal elliptic curves
        through $\Gamfour$,~\cite[Corollary~3.2]{Eisenbud__Popescu__Gale_geometry}.

        Each curve $E$ in this pencil is arithmetically Gorenstein and cut out by $5$
        quadrics. These quadrics define a rational map $c_E:\PP^4 \dashrightarrow
        \PP^4$ which is \emph{birational} with inverse $c'_E$  given by
        cubics~\cite{Crauder_Katz__Cremonas_with_smooth_fundamental_locus}.
        Now the image $c'_E(c_E(\PP^3))$ spans at most a $\PP^3$, hence the
        coordinates of $c'_E$ give a cubic equation between five of the
        quadrics defining $\PP(V) = \PP^3 \to \PP^6$. Such relation gives a
        cubic equation $C_E \in \PP\left(\Sym^3 L^{\perp}\right)$ of the
        image. The situation is summarised in the following diagram
        \begin{equation}\label{eq:cremonaEllipic}
            \xymatrix{\PP^3 \ar[r]^{\pi_L\circ v_2}\ar[d]^{lin}& \PP^6\ar@{-->}[d]^{lin}\\
            \PP^4 \ar@{-->}[r]^{c_E} & \PP^4}
        \end{equation}
        The argument can be made global to obtain a
        map $syz:\PP^1 \to \PP\left(\Sym^3 L^{\perp}\right)$.

        It remains to show that there are at least \emph{three} linearly independent cubic
        equations. We will do this by showing that $syz$ is not constant or a
        line.
        This requires working with all curves from the pencil $\mathcal{E}$.
        These curves sweep out a rational cubic scroll which is the image of
        $\PP^2$. The ideal of this scroll is generated by three quadrics --- maximal minors
        of a $2\times 3$ matrix. Hence the containments between the quadrics are as follows:
        \[
            \{\mbox{$3$ in the ideal of scroll}\} \subset \{\mbox{$5$
            cutting out an elliptic curve}\} \subset \{\mbox{$7$
            defining the $8$-points}\}.
        \]
        Then any elliptic curve defines a point in $\Gr(2, 4)$ corresponding to
        the space of its quadric equations modulo the quadrics vanishing on
        the scroll. We obtain a map $eq:\PP^1 \to \Gr(2, 4)$.

        If $syz$ was constant, there would be a cubic relation between
        the three quadrics defining the scroll. But no such relation
        exists.
        Clearly $syz$ factors through $eq$. If the image of $syz$ was a line, also
        the image of $eq$ would be a line in the Pl\"ucker embedding. But such a line corresponds to a
        family of projective lines on a plane passing through a point. Such a point would
        give a quadric not in the ideal of the scroll but vanishing on all
        elliptic curves. This contradicts the fact that the elliptic curves fill the scroll.
    \end{proof}

    \begin{remark}\label{ref:dolgachev:rmk}
        The cubics $C_E$ appearing in the proof of
        Theorem~\ref{ref:schreyerfestdinner:thm} are called Segre
        cubics, see~\cite[p.~4]{Dolgachev__Segre_and_his_cubics}.
    \end{remark}

    One wonders to what extent the statement of
    Theorem~\ref{ref:schreyerfestdinner:thm} can be reversed. For this
    consider first the $\PP(L) \subset \PP(\Sym^2 V)$ which intersect the
    secant (then $\Sym^3 L^{\perp}\to \Sym^6 V^\vee$ is not an isomorphism).
    \begin{lem}\label{ref:secant:lem}
        Let $\PP(L)$ be a projective plane intersecting the second secant of
        $v_2(\PP(V))$. Then $X_L$ lies on an at least three dimensional
        space of cubics.
    \end{lem}
    \begin{proof}
        The secant is the locus of quadrics of rank at most two. The jump
        locus is closed, so we may assume that $\PP(L)$ intersects the secant
        in a rank two quadric $xy$. Then no form of $L^{\perp}$ contains the
        monomial $\partial_x \partial_y$. Enlarge $x$, $y$ to a basis
        $x$, $y$, $z$, $t$ of $V$. Then $L^{\perp} \subset \spann{\partial_x^2,
        \partial_y^2} + \spann{\partial_z, \partial_t}\cdot V^{\vee}$.
        Observe that $(L^{\perp})^3 \subset \Sym^6 V^{\perp}$ annihilates the forms $x^5y$, $x^3y^3$
        and $xy^5$. Hence,
        \[
            \dim((L^{\perp})^3) \leq \dim(\Sym^6 V^{\perp})-3 = \dim(\Sym^3 L^{\perp}) - 3
        \]
        and we obtain an (at least) three
        dimensional space of cubics containing $X_L$.
    \end{proof}
    \begin{example}
        Note that when $L \subset \Sym^2 W$ for some three-dimensional subspace $W
        \subset V$, then $\PP(L)$
        insersects the secant variety of $v_2(\PP(W))$, because the latter is
        a divisor in $\PP(\Sym^2 W)$. Thus, $\PP(L)$ intersects
        the secant variety of $v_2(\PP(V))$.
    \end{example}
    Now we prove a stronger form of Theorem~\ref{ref:schreyerfestdinner:thm}.
    For the definition of jump condition, see the paragraph above
    Theorem~\ref{ref:schreyerfestdinner:thm}.
    \begin{thm}\label{ref:schreyerfestdinnerimproved:thm}
        Let $[\PP(L)]\in \Gproj$ be a projective plane.
        The following conditions are equivalent:
        \begin{enumerate}
            \item\label{it:ptone} $\PP(L)$ satisfies the jump condition,
            \item\label{it:pttwo} $\PP(L)$ satisfies the jump condition and the jump is at
                least by three,
            \item\label{it:ptthree} either $\PP(L)$ intersects the secant variety of $v_2(\PP(V))$
                or the zero-dimensional scheme associated to $L$ is
                smoothable.
        \end{enumerate}
    \end{thm}
    \begin{proof}
        Clearly,~\eqref{it:pttwo} implies~\eqref{it:ptone}. By
        Theorem~\ref{ref:schreyerfestdinner:thm} and
        Lemma~\ref{ref:secant:lem} we have~\eqref{it:ptthree}
        implies~\eqref{it:pttwo}. It remains to prove that~\eqref{it:ptone}
        implies~\eqref{it:ptthree}. These conditions are both divisorial (for
        smoothability, see Section~\ref{ssec:HilbertEight}), so it remains to
        check their degrees.

        The (Pl\"ucker) degree of the locus of spaces intersecting the secant
        variety is $10$ and the degree of the smoothable divisor is $2$,
        see~\eqref{eq:smoothableDivisorEquation}.  By the discussion
        below  Map~\eqref{eq:degenerationGlobal} (before Theorem \ref{ref:schreyerfestdinner:thm}) , the jump locus has degree $36$.
        The jump on the secant and smoothable parts is both at least by three,
        so these divisors contribute to the jump locus with multiplicity at
        least three, hence in total they contribute by $3\cdot (10 + 2) = 36$
        to the degree. Thus the jump locus is equal to the union of those
        divisors.
    \end{proof}

    \subsection{Cremona transformations associated to octic
    surfaces}\label{ssec:Cremonas}
   Let us describe further the geometry of the Cremona transformations described in the proof of
    Theorem~\ref{ref:schreyerfestdinner:thm}. We claim that these transformations can be lifted to the
    Cremona transformations defined by octic surfaces, which we now recall.
    Consider a set $Z \subset \PP^2$ of eight points. The quartics through
    $Z$ define an embedding $\Bl_Z \PP^2 \to \PP^6$. The embedded
    surface $S_8 = S_8(Z) \subset \PP^6$ is a rational \emph{octic surface}.
    One checks that $S_8$ is contained in a singular scroll $X$ and in
    fact $S_8 = D_1 \cap D_2$ is a complete intersection of two linearly
    equivalent divisors inside the
    scroll~\cite[Section~3]{hulek_katz_schreyer__Cremonas_and_syzygies}.

    Moreover~\cite[Thm~3.2]{hulek_katz_schreyer__Cremonas_and_syzygies}, the
    quadrics through $S_8$ define a Cremona transformation $c_{S_8}\colon\PP^6
    \dashrightarrow \PP^6$, whose inverse is given by quartics through a model
    of $\PP^4$ blown in $8$ points~\cite{Semple_Tyrrell__SpecialCremonas}.
    If one takes a general $\PP^4$-section in $\PP^6$ then $S_8\cap \PP^4$
    is a tuple of eight points lying on a pencil of elliptic normal curves
    generated by $D_1 \cap \PP^4$ and $D_2\cap \PP^4$ and this pencil
    fills a cubic scroll $X\cap \PP^4$.  Each Cremona transformation related to an
    elliptic curve in this pencil is obtained from $\PP^6\dashrightarrow
    \PP^6$ by composing with a linear embedding and a linear projection.
    In consequence we obtain a further geometric description of our projection from $L$.

    \begin{prop}\label{ref:embeddingInOctic:prop}
        Let $\PP(L) \subset \PP(\Sym^2 V)$ be a $\PP^2$ corresponding to a
        general element of the smoothable divisor.
        Then there exists an octic
        surface $S_8$ and a linear embedding $\ell\colon\PP^3 \to \PP^6$ such that
        $\pi_L = \ell \circ c_{S_8}$:
        \[
            \xymatrix{\PP^3 \ar[r]^{\pi_L\circ v_2}\ar[d]^{lin}& \PP^6\ar@{=}[d]\\
            \PP^6 \ar@{-->}[r]^{c_{S_8}} & \PP^6}
        \]
    \end{prop}

    \begin{proof}
        Since $\PP(L)$ corresponds to a general element of the smoothable
        divisor, there is a (suitably general) set of eight points
        $\Gamma\subset \PP^4$ such that $\pi_L\circ v_2$ lifts to a map $\PP^4
        \dashrightarrow \PP^6$ given by quartics
        through $\Gamma$, see
        Corollaries~\ref{ref:smoothabilityofgraded:cor}-\ref{ref:extendingrationas:cor}.
        The claim follows if we prove that $\Gamma$ is a linear section of an
        octic surface $S_8$, which we do below.
        Let $\Gamma'\subset \PP^2$ be the Gale transform of $\Gamma\subset \PP^4$.
        By Gale
        duality~\cite[Corollary~3.2]{Eisenbud__Popescu__Gale_geometry}, the
        cubics containing $\Gamma'\subset \PP^2$ give a pencil of
        elliptic normal curves containing $\Gamma$, which
        fill a smooth cubic scroll $X\subset \PP^4$ containing $\Gamma$.
        The Picard group of $X$ is generated by the hyperplane section of the scroll $H$ and the fiber $R$
        of the scroll. From the adjunction formula, the pencil of elliptic curves is a subsystem the system $|2H-R|$.
        Let $D_1$, $D_2$ be two elements from this subsystem. Since $D_1.D_2=8$ we infer that $D_1\cap D_2=\Gamma$.

        Let us consider cubic scrolls $Y\subset \PP^5$ and $Z\subset \PP^6$ such that
        $Z$ restricts to $Y$ and $Y$ to $X$. We can assume that $Y$ is smooth and $Z$
        is a cone over $Y$.  From the restriction exact sequence $$0\to
        \OO_Y(H-R)\to\OO_Y(2H-R)\to \OO_X(2H-R)\to 0$$ and the fact that
        $h^1(\OO_Y(H-R))=h^1(\OO_Z(H-R))=0$ we find divisors $B_1$
        and $B_2$ on $Z$ in the linear system $|2H-R|$ such that $B_1|_X=D_1$ and
        $B_2|_X=D_2$ (in fact we have a freedom of choice of $B_1$ and $B_2$).
        Then $B_1\cap B_2\cap \PP^4 = \Gamma$, so $B_1\cap B_2$ is a complete intersection.
        By~\cite[Section~3]{hulek_katz_schreyer__Cremonas_and_syzygies}, the
        intersection of two divisors from $|2H-R|$ is a rational octic surface $S_8\subset
        \PP^6$. Now, $\Gamma$ is the intersection of this $S_8$ with the
        linear subspace spanned by $X$.
    \end{proof}
    \begin{remark}
        By Proposition~\ref{ref:embeddingInOctic:prop} the elliptic curves $E$ constructed in the proof of
        Theorem~\ref{ref:schreyerfestdinner:thm} appear as linear sections of
        $S_8$ and the corresponding Cremona transformations $c_E$ come from
        $c_{S_8}$ composed with a linear embedding and projection. Note that
        $c_{S_8}$ has type $(2, 4)$ while all $c_E$ have type $(2, 3)$.
    \end{remark}
    Now we formally summarize how do the previous results add up to give our
    main result.
    \begin{proof}[Proof of Theorem~\ref{ref:mainintro:thm}]
        First, since $L$ does not intersect the secant of $v_2(\PP(V))$, by
        Subsection~\ref{ssec:HilbertEight} we have $\Apolar{L} =
        \Apolar{L+V}$.
        Let us recapitulate the proof of equivalence of
        Conditions~\ref{it:maincremona}-\ref{it:maindescription}. The
        equivalence of~\ref{it:mainsmoothable}
        and~\ref{it:maindescription} is proven
        in~\ref{ref:cubicvssmoothable:prop}. The equivalence of
        Conditions~\ref{it:mainjump}, \ref{it:mainjumpthree} and
        \ref{it:mainsmoothable} is given in Theorem
        \ref{ref:schreyerfestdinnerimproved:thm}. By
        Proposition~\ref{ref:embeddingInOctic:prop},
        Condition~\ref{it:mainsmoothable} implies Condition~\ref{it:maincremona}. Finally,
        if $\pi_L \circ v_2$ is a general restriction of a Cremona based in a rational
        octic surface, then $\Spec(\Apolar{L})$ is a section of a cone over
        $8$ points of that surface, hence is smoothable. Thus,
        Condition~\ref{it:maincremona} implies~\ref{it:mainsmoothable} and the
        proof of equivalences is concluded. The irreducibility of the
        smoothable divisor follows from the parametrization
        from~\ref{it:maindescription} and was proven in~\cite{CEVV}. For a general $L$ in this
        divisor, the
        existence of a three-dimensional space of Segre cubics follows from
        Remark~\ref{ref:dolgachev:rmk}. If $I(X_L)_3$ was
        at least four-dimensional for all $L$ in the smoothable divisor, then the contribution of
        the smoothable divisor to the jump locus would be higher than allowed,
        see the proof of Theorem~\ref{ref:schreyerfestdinnerimproved:thm}.
    \end{proof}


\subsection*{Acknowledgements}
The work was supported by the Polish National Science Center project number:
2013/10/E/ST1/00688. We thank K.~Ranestad and F.~Russo for helpful comments.
We thank the organizers of the Schreyer 60 conference where this work has
begun. A final revision of this work was done during the Simons Semester
\emph{Varieties: Arithmetic and Transformations} in Warsaw Sep 1-Dec 1, 2018.

\small
\newcommand{\etalchar}[1]{$^{#1}$}
\def\cdprime{$''$}

\normalsize

\vspace*{1cm}
\noindent IMPAN Warsaw, \\ email: jjelisiejew@impan.pl\\
\\
 Jagiellonian University Cracow, Department of Mathematics \\ email: grzegorz.kapustka@uj.edu.pl \\
\\
IMPAN Warsaw, \\ email: michal.kapustka@impan.pl\\
 Jagiellonian University Cracow, Department of Mathematics, michal.kapustka@uj.edu.pl \\
 University of Stavanger, Department of Mathematics and Natural Sciences, \\email: michal.kapustka@uis.no

\end{document}